\documentclass[12pt]{amsart}
\title{Topological proofs of results on large fields}
\author{Erik Walsberg\vspace{-6ex}}
\address{Department of Mathematics\\ University of California, Irvine}
\email{ewalsber@uci.edu}
\urladdr{https://www.math.uci.edu/\textasciitilde ewalsber}

\usepackage{amsmath, amssymb, amsthm, enumitem, comment}    
\usepackage[mathscr]{euscript}
\usepackage{fullpage} 	
\usepackage{hyperref}
\usepackage{lineno}
\usepackage[all]{xy}
\usepackage{centernot}
\usepackage{xcolor}

\DeclareFontFamily{U}{BOONDOX-calo}{\skewchar\font=45 }
\DeclareFontShape{U}{BOONDOX-calo}{m}{n}{
  <-> s*[1.05] BOONDOX-r-calo}{}
\DeclareFontShape{U}{BOONDOX-calo}{b}{n}{
  <-> s*[1.05] BOONDOX-b-calo}{}
\DeclareMathAlphabet{\mathcalboondox}{U}{BOONDOX-calo}{m}{n}
\SetMathAlphabet{\mathcalboondox}{bold}{U}{BOONDOX-calo}{b}{n}
\DeclareMathAlphabet{\mathbcalboondox}{U}{BOONDOX-calo}{b}{n}

\DeclareMathOperator*{\forkindep}{\raise0.2ex\hbox{\ooalign{\hidewidth$\vert$\hidewidth\cr\raise-0.9ex\hbox{$\smile$}}}}

\newcommand{\Sa}[1]{\ensuremath{\mathscr{#1}}}

\newcommand{\trdg}{\operatorname{td}}

\newcommand{\pac}{\mathrm{PAC}}

\newcommand{\lins}{L^{\mathrm{ins}}}

\newcommand{\Spec}{\operatorname{Spec}}

\newcommand{\Chara}{\operatorname{Char}}

\newtheorem*{claim-star}{Claim}
\newtheorem{theorem}{Theorem}[section] 
\newtheorem{lemma}[theorem]{Lemma}

\newtheorem{prop-def}[theorem]{Proposition-Definition}

\newtheorem{fact}[theorem]{Fact}
\newtheorem{fact-eh}[theorem]{Fact(?)}

\newtheorem{question}[theorem]{Question}
\newtheorem{proposition}[theorem]{Proposition}
\newtheorem{proposition-eh}[theorem]{Proposition(?)}
\newtheorem*{theorem-star}{Theorem}
\newtheorem*{conjecture-star}{Conjecture}
\newtheorem*{lemma-star}{Lemma}

\newtheorem*{factA}{Fact A}
\newtheorem*{factB}{Fact B}
\newtheorem*{factC}{Fact C}
\newtheorem*{factCC}{Alternate form of Fact C}
\newtheorem*{propA}{Proposition A}
\newtheorem*{propB}{Proposition B}
\newtheorem*{propC}{Proposition C}

\theoremstyle{definition}

\theoremstyle{remark}

\newcommand{\Aa}{\mathbb{A}}

\newcommand{\Qq}{\mathbb{Q}}
\newcommand{\Rr}{\mathbb{R}}
\newcommand{\Zz}{\mathbb{Z}}

\newcommand{\meno}{\medskip \noindent}

\newenvironment{claimproof}[1][\proofname]
               {
                 \proof[#1]
                 
               }
               {
                 \endproof
               }

\begin{document}
\maketitle
\begin{abstract}
We use the recently introduced \'etale open topology to prove several facts on large fields.
We show that these facts lift to a very general topological setting.
\end{abstract}

Throughout $K,L$ are fields, $L$ is infinite, and $\Aa^m_K, \Aa^m_L$ is $m$-dimensional affine space over $K,L$, respectively.
A $K$-variety is a separated $K$-scheme of finite type, not assumed to be reduced.
If $K$ is a subfield of $L$ and $V$ is a $K$-variety then $V_L = V \times_{\Spec K} \Spec L$ is the base change of $V$, and if $f : V \to W$ is a morphism of $K$-varieties then $f_L : V_L \to W_L$ is the base change of $f$.
Given a $K$-variety $V$ we let $V(K)$ be the set of $K$-points of $V$, $K[V]$ be the coordinate ring of $V$, and $K(V)$ be the function field of $V$ when $V$ is integral.

\meno
$L$ is \textbf{large} if every smooth $L$-curve with an $L$-point has infinitely many $L$-points.
Finitely generated fields are not large.
Most other fields of particular interest are either large, or are function fields over large fields, or have unknown status.
Local fields, real closed fields, separably closed fields, fields which admit Henselian valuations, quotient fields of Henselian domains, pseudofinite fields, infinite algebraic extensions of finite fields, $\pac$ fields, $p$-closed fields, and fields that satisfy a local-global principle are all large.
Function fields are not large.
It is an open question whether the maximal abelian or maximal solvable extension of $\Qq$ is large.
See \cite{Pop-little} and \cite{open-problems-ample} for more background on large fields.

\meno
We will give topological proofs of Facts A,B, and C below.
Fact A is \cite[Proposition 2.6]{Pop-little}.

\begin{factA}
Suppose that $L$ is large and $V$ is an irreducible $L$-variety with a smooth $L$-point.
Then $V(L)$ is Zariski dense in $V$.
\end{factA}

Facts B and C are due to Fehm.
Fact B is proven in \cite{Fehm-subfield}.
Note that Fehm uses ``ample" for ``large" (this is one of a surprisingly large number of names used in the literature.)

\begin{factB}
\label{fact:fehm}
Suppose that $L$ is large, $K$ is a proper subfield of $L$, and $V$ is a positive-dimensional irreducible $K$-variety with a smooth $K$-point.
Then $|V(L) \setminus V(K)| = |L|$.
\end{factB}


Fact B is a strengthening of the fact that if $L$ is large and $V$ is a positive-dimensional irreducible $L$-variety with a smooth $L$-point then $|V(L)| = |L|$.
This was previously proven by Pop, see \cite[Proposition 3.3]{harbater}.
We give a separate proof of this fact in Section~\ref{section:very-large}.
Secondly, Fact B and the fact that an algebraic extension of a large field is large yields the following: if $K$ is large, $V$ is a positive dimensional irreducible $K$-variety with a smooth $K$-point, and $L/K$ is algebraic then $|V(L) \setminus V(K)| =  |L|$.
(We also give a topological proof of the fact that large fields are closed under algebraic extensions in Section~\ref{section:fin-ext}.)

\meno
Fact C is from \cite{fehm-embedding}.
We let $\trdg(E/F)$ be the transendence degree of a field extension $E/F$.

\begin{factC}
\label{fact:embedding}
Suppose $K$ is a subfield of $L$, $L$ is large, and $V$ is a smooth geometrically integral $K$-variety.
Then the following are equivalent:
\begin{enumerate}
\item $\trdg(L/K) \ge \dim V$ and $V(L) \ne \emptyset$,
\item there is a $K$-algebra embedding $K(V) \to L$.
\end{enumerate}
\end{factC}

Another proof of Fact C is given in \cite[Proposition 1.1]{large-diff-galois}, they reduce to the one-dimensional case which then follows directly by Fact B.
The implication $(2) \Rightarrow (1)$ is routine and does not require largeness.
We describe a geometric statement equivalent to $(1) \Rightarrow (2)$.
Suppose that $p \in V(L)$ and $p \notin W(L)$ for any proper closed subvariety $W$ of $V$.
Let $U$ be an affine open subvariety of $V$, so $p \in U(L)$.
Note that $K(U) = K(V)$ and $K(V)$ is the fraction field of $K[U]$.
Now $p$ gives a morphism $\Spec L \to U$, which is dual to an $K$-algebra morphism $K[U] \to L$.
Note that $K[U] \to L$ is injective as $p \notin W(L)$ for any proper closed subvariety $W$ of $V$.
So $K[U] \to L$ extends to a $K$-algebra morphism $K(V) = K(U)  \to L$.
So we prove the following.

\begin{factCC}
Suppose that $L$ is large, $K$ is a subfield of $L$ with $\trdg(L/K) \ge m$, and $V$ is a smooth geometrically integral $m$-dimensional $K$-variety with $V(L) \ne \emptyset$.
Then there is $p \in V(L)$ such that $p \notin W(L)$ for any proper closed subvariety $W$ of $V$.
\end{factCC}

We now discuss our proof technique.
Each fact says that $V(L)$ is large in some sense.
Fix a smooth $p \in V(L)$.
There is an open subvariety $U$ of $V$ containing $p$ and an \'etale morphism $f : U \to \Aa^m_L$, and $f(U(L))$ is a nonempty \'etale open subset of $L^m$.
This allows us to reduce each of the facts above to a statement saying non-empty \'etale open subsets of $L^m$ are large in some sense.
In each case the statement holds in a very broad setting which we now describe.

\meno
A \textbf{system of topologies} $\Sa T$ over $L$ is a choice of topology on $V(L)$ for each $L$-variety $V$ such that the following holds for any morphism $f : V \to W$ of $L$-varieties:
\begin{enumerate}
\item the induced map $V(L) \to W(L)$ is continuous,
\item if $f$ is an open immersion then $V(L) \to W(L)$ is a topological open embedding, and
\item if $f$ is a closed immersion then $V(L) \to W(L)$ is topological closed embedding.
\end{enumerate}

If $\tau$ is a Hausdorff field topology on $L$ then we produce a system of topologies by equipping each $V(L)$ with the usual $\tau$-topology, the other familiar example of a system is the Zariski topology.
The \'etale open topology is a system of topologies, which may or may not be induced by a Hausdorff field topology on $L$.
It is easy to see that the $\Sa T$-topology on $L = \Aa^1_L(L)$ is discrete if and only if the $\Sa T$-topology on $V(L)$ is discrete for every $L$-variety $V$, and we say that $\Sa T$ is \textbf{discrete} if these conditions hold.
We show in \cite{firstpaper} that $L$ is large if and only if the \'etale open topology over $L$ is not discrete.

\meno
Fact A,B,C follows from Proposition A,B,C, respectively.

\begin{propA}
Suppose that $\Sa T$ is a non-discrete system of topologies over $L$ and $O$ is a nonempty $\Sa T$-open subset of $L^m$.
Then $O$ is Zariski dense in $\Aa^m_L$.
\end{propA}

\begin{propB}
Suppose that $\Sa T$ is a non-discrete system of topologies over $L$, $K$ is a proper subfield of $L$, and $O$ is a nonempty $\Sa T$-open subset of $L^m$.
Then $|O \setminus K^m| = |L|$.
\end{propB}

Note that if $a = (a_1,\ldots,a_m) \in L^m$ then $\trdg(K(a_1,\ldots,a_m)/K)$ is the minimum dimension of a closed subvariety $W$ of $\Aa^m_K$ such that $a \in W(L)$.

\begin{propC}
Suppose that $\Sa T$ is a non-discrete system of topologies over $L$, $K$ is a subfield of $L$ with $\trdg(L/K) \ge m$, and $O$ is a nonempty $\Sa T$-open subset of $L^m$.
Then there is $(a_1,\ldots,a_m) \in O$ such that $\trdg(K(a_1,\ldots,a_m)/K) = m$.
Equivalently there is $a \in O$ such that $a \notin W(L)$ for any proper closed subvariety $W$ of $\Aa^m_K$.
\end{propC}

In Sections~\ref{section:finite} and \ref{section:diophantine} we use unpublished work of JTWY to give topological proofs of two more facts.
In this case our proof is specific to the \'etale open topology and does not yield a more general result on systems of topologies.

\subsection{Acknowledgements}
The basic facts about the \'etale open topology that we use were developed jointly with Will Johnson, Chieu-Minh Tran, and Vincent Ye.
The ideas in the proof of Fact B come from work of Arno Fehm.
Arno Fehm also read an earlier version of this note, made helpful suggestions, and pointed out mistakes.

\section{Background}
\label{section:background}
It is worth noting that any system of topologies refines the Zariski topology, i.e. if $\Sa T$ is a system of topologies over $L$ and $V$ is an $L$-variety then the $\Sa T$-topology on $V(L)$ refines the Zariski topology.
Fact~\ref{fact:proj} is proven in \cite{firstpaper}.
The $\Sa T$-topology and the product of the $\Sa T$-topologies on $(V \times W)(L) = V(L) \times W(L)$ may not agree.

\begin{fact}
\label{fact:proj}
Suppose that $\Sa T$ is a system of topologies over $L$ and $V,W$ are $L$-varieties.
Then the projection $V(L) \times W(L) \to V(L)$ is a $\Sa T$-open map.
\end{fact}

We will also make frequent use the obvious fact that the $\Sa T$-topology on $L$ is affine invariant, i.e. the map $L \to L$, $x \mapsto ax + b$ is a homeomorphism for all $a \in L^\times, b \in L$.
In particular this implies that $\Sa T$ is discrete if and only if there is a non-empty finite $\Sa T$-open subset of $L$.

\meno
Let $V$ be a $L$-variety.
An \textbf{\'etale image} in $V(L)$ is a set of the form $h(W(L))$ for an \'etale morphism $h : W \to V$ of $L$-varieties.
We emphasize that Fact~\ref{fact:basic-system} follows from standard facts on \'etale morphisms.
Fact~\ref{fact:basic-system} is also proven in \cite{firstpaper}.

\begin{fact}
\label{fact:basic-system}
Given an $L$-variety $V$, the collection of \'etale images in $V(L)$ is a basis for a topology.
The collection of such topologies forms a system of topologies over $L$.
If $f : V \to W$ is an \'etale morphism of $L$-varieties and $O$ is an \'etale open subset of $V(L)$ then $f(O)$ is an \'etale open subset of $W(L)$.
\end{fact}

We refer to this system of topologies as the \textbf{\'etale open topology (over $L$)}.
We are not aware of any direct connection to the well-known \'etale topology.
We will sometimes refer to it as the $\Sa E_L$-topology when there are multiple fields in play.


\section{Algebraic extensions}
\label{section:fin-ext}

Fact~\ref{fact:finite-extension} is \cite[Proposition 2.7]{Pop-little}.

\begin{fact}
\label{fact:finite-extension}
If $K$ is a subfield of $L$, $K$ is large, and $L/K$ is algebraic, then $L$ is large.
\end{fact}

There is a field $K$ and a finite extension $L/K$ such that $L$ is large and $K$ is not large~\cite{Srinivasan}.
Fact~\ref{fact:down} is \cite[Theorem 4.10]{firstpaper}.
The proof does not make use of largeness.
(The proof of Fact~\ref{fact:down}, like all other proofs of Fact~\ref{fact:finite-extension}, uses a form of Weil restriction.)

\begin{fact}
\label{fact:down}
Suppose that $K$ is a subfield of $L$, $L/K$ is algebraic, and $V$ is a $K$-variety.
The $\Sa E_K$-topology on $V(K)$ refines the topology induced by the $\Sa E_L$-topology on $V_L(L) = V(L)$.
\end{fact}

We view Fact~\ref{fact:down} as a topological refinement of Fact~\ref{fact:finite-extension}. 
We prove Fact~\ref{fact:finite-extension}.

\begin{proof}
Suppose that $L/K$ is algebraic and $L$ is not large.
Then the $\Sa E_L$-topology on $L$ is discrete, so $\{0\}$ is an $\Sa E_L$-open subset of $L$.
By Fact~\ref{fact:down} $\{0\} = \{0\} \cap K$ is an $\Sa E_K$-open subset of $K$.
So the $\Sa E_K$-topology on $K$ is discrete, so $K$ is not large.
\end{proof}

\section{Fact A}
\label{section:density}

We first prove Proposition A.

\begin{proof}
Suppose that $O$ is not Zariski dense in $\Aa^m_L$ and let $W$ be the Zariski closure of $U$ in $\Aa^m_L$.
So $\dim W < m$.
Fix $p \in O$.
A typical line in $\Aa^m_L$ passing through $p$ will intersect $W$ in only finitely many points.
So there is a closed immersion $g : \Aa^1_L \to \Aa^m_L$ such that $g(0) = p$ and $g(\Aa^1_L) \cap W$ is finite.
Let $O'$ be the preimage of $O$ under the induced map $L \to L^m$.
So $O'$ is a nonempty finite $\Sa T$-open subset of $L$, hence $\Sa T$ is discrete, contradiction.
\end{proof}

We now prove the following stronger version of Fact A.

\begin{proposition}
\label{prop:A-strong}
Suppose that $L$ is large, $V$ is an irreducible $L$-variety, and $O$ is an \'etale open subset of $V(L)$ which contains a smooth $L$-point.
Then $O$ is Zariski dense in $V$.
\end{proposition}

\begin{proof}
Fix a smooth $p \in O$ and let $m = \dim V$.
The case $m = 0$ is trivial so we suppose $m \ge 1$.
Fix an open subvariety $U$ of $V$ containing $p$ and an \'etale morphism $f : U \to \Aa^m_L$.
Let $P = f(U(L) \cap O)$, so $P$ is a non-empty \'etale open subset of $L^m$.
Suppose that $O$ is not Zariski dense in $V$ and let $W$ be the Zariski closure of $O$ in $V$.
Then $\dim W < m$ hence $\dim U \cap W < m$, and the Zariski closure of $f(U \cap W)$ has dimension $< m$.
Therefore $P \subseteq f(U \cap W)$ is not Zariski dense in $\Aa^m_L$, contradiction.
\end{proof}

\section{Many $L$-points}
\label{section:very-large}
Before proving Fact B we prove the following related result.

\begin{proposition}
\label{prop:very-large}
Suppose that $L$ is large, $V$ is an irreducible $L$-variety, and $O$ is a nonempty \'etale open subset of $V(L)$ which contains a smooth $L$-point.
Then $|O| = |L|$.
\end{proposition}

Proposition~\ref{prop:very-large} follows from a more general fact.

\begin{proposition}
\label{prop:very-large-gen}
Suppose that $\Sa T$ is a non-discrete system of topologies over $L$ and $O$ is a nonempty $\Sa T$-open subset of $L^m$.
Then $|O| = |L|$.
\end{proposition}

\begin{proof}
Let $\pi : L^m \to L$ be a coordinate projection.
By Fact~\ref{fact:proj} $\pi(O)$ is $\Sa T$-open.
As $|O| \ge |\pi(O)|$ we may suppose that $m = 1$.
\begin{claim-star}
If $O$ contains $0$ then $OO^{-1} = L$, hence $|O| = |L|$.
\end{claim-star}
\begin{claimproof}
Suppose that $O$ contains $0$ and $OO^{-1} \ne L$.
Fix $a \in L^\times \setminus OO^{-1}$.
Then $O \cap aO = \{0\}$.
However, $O \cap aO$ is $\Sa T$-open, so $\Sa T$ is discrete, contradiction.
\end{claimproof}
Note that if $b \in O$ then $|O| = |O - b| = |L|$.
\end{proof}

We now prove Proposition~\ref{prop:very-large}.

\begin{proof}
Suppose that $p \in O$ is smooth.
Let $m = \dim V$, $U$ be an open subvariety of $V$, and $f : U \to \Aa^m_L$ be an \'etale morphism.
Then $f(U(K) \cap O)$ is a nonempty \'etale open subset of $L^m$.
Apply Proposition~\ref{prop:very-large-gen}.
\end{proof}

\section{Fact B}
\label{section:new-points}
We will need a couple lemmas.
Fact~\ref{fact:linear-algebra} is a special case of \cite[Lemma 3]{Fehm-subfield}.

\begin{fact}
\label{fact:linear-algebra}
Suppose that $F$ is a field, $S$ is an $F$-vector space of dimension $\ge 2$, $I$ is an index set, and $S_i$ is a one-dimensional subspace of $S$ for all $i \in I$.
If $S = \bigcup_{i \in I} S_i$ then $|I| \ge |F|$.
\end{fact}

\begin{proof}
Fix a one-dimensional subspace $S'$ of $S$ and $a \in S \setminus S'$.
Then $|S'| = |F|$ and it is easy to see that $|S_i \cap (a + S')| \le 1$ for all $i \in I$.
\end{proof}

Lemma~\ref{lem:large-difference} is essentially in the proof of \cite[Lemma 4]{Fehm-subfield}.
Recall our standing assumption that $L$ is infinite.

\begin{lemma}
\label{lem:large-difference}
Suppose that $K$ is a proper subfield of $L$ and $X \subseteq K$ satisfies $XX^{-1} = L$.
Then $| X \setminus K| = |L|$.
\end{lemma}

\begin{proof}
Note that $|X| = |L|$.
Suppose that $|K| < |L|$.
Then $|X| = |L| > |K|$ so $|X \setminus K| = |L|$.
So we may suppose that $|K| = |L|$.
It now suffices to show that $|X \setminus K| \ge |K|$.
We let $Y^\times = Y \setminus \{0\}$ for any $Y \subseteq L$.
Let $A = X \cap K$ and $B = X \setminus K$.
So
\begin{align*}
L &= \{ a/b : a \in X, b \in X^\times \}\\
   &= \{ a/b : a \in A, b \in A^\times \} \cup \{ a/b : a \in A, b \in B^\times \} \cup \{ a/b : a \in B, b \in A^\times \} \cup \{ a/b : a \in B, b \in B^\times \} \\
&\subseteq K \cup \left(\bigcup_{b \in B^\times} (1/b)K \right) \cup \left(\bigcup_{a \in B} a K\right) \cup \left(\bigcup_{a \in B, b \in B^\times} (a/b) K\right). 
\end{align*}
Consider $L$ to be a $K$-vector space.
So $L$ is a union of $\le 1 + 2|B| + |B|^2$ one-dimensional subspaces.
By Fact~\ref{fact:linear-algebra} we have $1 + 2|B| + |B|^2 \ge |K|$.
As $K$ is infinite $|B| \ge |K|$.
\end{proof}

We now prove Proposition B.

\begin{proof}
Let $\pi : L^m \to L$ be the projection onto the first coordinate.
By Fact~\ref{fact:proj} $\pi(O)$ is $\Sa T$-open.
We have $\pi(O) \setminus K \subseteq \pi(O \setminus K^m)$, so it suffices to show that $|\pi(O) \setminus K| = |L|$.
So we may suppose that $O$ is an $\Sa T$-open subset of $L$.
By the proof of Proposition~\ref{prop:very-large-gen} $(O - b)(O - b)^{-1} = L$ for any $b \in O$.
So $|O| = |L|$.
So if $O \cap K = \emptyset$ we are done.
Suppose otherwise and fix $b \in O \cap K$.
By the claim $(O - b)(O - b)^{-1} = L$ so by Lemma~\ref{lem:large-difference} $|(O - b) \setminus K| = |L|$.
Note that $x \mapsto x + b$ gives a bijection $(O - b) \setminus K \to O \setminus K$.
\end{proof}

We now prove Fact B.

\begin{proof}
Let $p$ be a smooth $K$-point of $V$ and $m = \dim V$.
As $V$ is irreducible there is an open subvariety $U$ of $V$ containing $p$ and an \'etale morphism $f : U \to \Aa^m_K$.
Let $O = f_L(U_L(L))$, note that $f_L$ is \'etale as \'etale morphisms are closed under base change.
Then $O$ is an \'etale image in $\Aa^m_L(L) = L^m$ and is hence \'etale open.
By Proposition B $|U \setminus K^m| = |L|$.
Note that if $p \in U(K)$ then $f_L(p) = f(p) \in K^m$.
\end{proof}

\newpage
\section{Fact C}
We now prove Proposition C.
Given $a = (a_1,\ldots,a_m) \in L^m$ we let $K(a) = K(a_1,\ldots,a_m)$.

\begin{proof}
We apply induction on $m$.
Suppose $m = 1$.
Let $K'$ be the algebraic closure of $K$ in $L$.
So $K'$ is a proper subfield of $L$.
By Proposition B there is $a \in U \setminus K'$.
So $\trdg(K(a)/K) = 1$.
Suppose $m \ge 2$.
Let $\pi : K^m \to K^{m - 1}$ be the projection away from the first coordinate.
By Fact~\ref{fact:proj} $\pi(U)$ is $\Sa T$-open.
By induction there is $b \in \pi(U)$ such that $\trdg(K(b)/K) = m - 1$.
Let $U_b = \{ c \in L : (b,c) \in U \}$.
Note that $U_b$ is the pre-image of $U$ under the map $K \to K^m$ given by $x \mapsto (b,x)$.
So $U_b$ is $\Sa T$-open.
As $\trdg(L/K) \ge m$ we have $\trdg(L/K(b)) \ge 1$.
So there is $c \in U_b$ such that $\trdg(K(b,c)/K(b)) = 1$.
Let $a = (b,c)$.
\end{proof}

We now prove a stronger version of the second form of Fact C.

\begin{proposition}
\label{prop:C-strong}
Suppose that $L$ is large, $K$ is a subfield of $L$ with $\trdg(L/K) \ge m$, and $V$ is a smooth geometrically irreducible $m$-dimension $K$-variety.
Then the set of $p \in V(L)$ such that $p \notin W(L)$ for any proper closed subvariety $W$ of $V$ is \'etale open dense in $V(L) = V_L(L)$.
\end{proposition}

\begin{proof}
Suppose that $O$ is a nonempty \'etale open subset of $V(L)$.
As $V$ is smooth and irreducible there is an open subvariety $U$ of $V$ and an \'etale morphism $f : U \to \Aa^m_K$.
By Proposition~\ref{prop:A-strong} $O$ intersects $U_L(L)$.
Let $P = f_L(U_L(L) \cap O)$, so $P$ is nonempty \'etale open subset of $L^m$.
By Proposition C there is $a \in P$ such that $\trdg(K(a)/K) = m$.
Fix $p \in O \cap U_L(L)$ such that $f_L(p) = a$.
Suppose $W$ is a proper closed subvariety of $V$, and let $W'$ be the Zariski closure of $f_L(W_L)$.
Then $\dim W < m$, so $\dim W_L < m$, so $\dim W' < m$.
Therefore $a \notin W'$, so $p \notin W_L(L) = W(L)$.
\end{proof}

\section{Images of finite morphisms}
\label{section:finite}

This striking question was posted on math overflow by Lampe in 2009 (Question 6820).

\begin{question}
\label{ques:lampe}
Suppose that $|K| \ge \aleph_0$, $f \in K[t]$, and $f(K) \ne K$.
Must $K \setminus f(K)$ be infinite?
\end{question}

This question was essentially asked by Reineke who conjectured that if every non-constant polynomial map $K \to K$ has cofinite image then $K$ is finite or algebraically closed, see \cite[Conjecture 6]{wagner-minimal-fields}.
A proof of Reineke's conjecture would answer Question~\ref{ques:lampe}.
The Reineke conjecture also implies the Podewski conjecture.
Fact~\ref{fact:kosters} is due to Kosters~\cite{kosters}.

\begin{fact}
\label{fact:kosters}
Suppose that $K$ is perfect and large and $f \in K[t]$ satisies $f(K) \ne K$.
Then $|K \setminus f(K)| = |K|$.
\end{fact}

Fact~\ref{fact:finite-image} is a generalization of Fact~\ref{fact:kosters} due to Bary-Soroker, Geyer, and Jarden~\cite{bsgj}.
We let $\lins$ be the maximal inseparable extension of $L$.

\begin{fact}
\label{fact:finite-image}
Suppose that $L$ is large, $f : V \to W$ is a finite morphism of irreducible $L$-varieties, and there is smooth $p \in W(L)$ such that $p \notin f(V(\lins))$.
Then $|W(L) \setminus f(V(\lins))| = |L|$.
In particular if $L$ is perfect then $|W(L) \setminus f(V(L))| = |L|$.
\end{fact}

We describe a topological proof of Fact~\ref{fact:finite-image}.
We cheat and use the following unpublished theorem of JTWY.

\begin{theorem}
\label{thm:unpub}
Suppose that $L$ is perfect, $f : V \to W$ is a finite morphism of $L$-varieties, and equip $W(L)$ with the \'etale open topology.
Then $f(V(L))$ is closed.
\end{theorem}

The assumption of perfection in Theorem~\ref{thm:unpub} is necessary.
Suppose that $L$ is separably closed and not algebraically closed.
It is shown in \cite{firstpaper} that the \'etale open topology over $L$ agrees with the Zariski topology.
The Frobenius is finite and the image of the Frobenius $L \to L$ is infinite and co-infinite, hence dense and co-dense in the Zariski topology.

\meno
We now prove Fact~\ref{fact:finite-image}.
As above we let $\Sa E_F$ be the \'etale open topology over a field $F$.\vspace{-.5ex}

\begin{proof}
As $\lins$ is perfect Theorem~\ref{thm:unpub} shows that $f(V_{\lins}(\lins))$ is an $\Sa E_{\lins}$-closed subset of $W_{\lins}(\lins)$.
By Fact~\ref{fact:down}
$$ [W_{\lins}(\lins) \setminus f(V_{\lins}(\lins) ] \cap W(L) =  W(L) \setminus f(V(\lins))$$
is an $\Sa E_L$-open subset of $W(L)$.
By assumption this $\Sa E_L$-open subset contains a smooth $L$-point of $W$.
An application of Proposition~\ref{prop:very-large} shows that $|W(L) \setminus f(V(\lins))| = |L|$.
\end{proof}

Fact~\ref{fact:koe} is due to  Koenigsmann, see the remarks after \cite[Conjecture 6.1]{open-problems-ample}.

\begin{fact}
\label{fact:koe}
Suppose $L$ is large, $f \in L[t]$ is irreducible, and $f(L) \ne L$.
Then $|L \setminus f(L)| = |L|$.
\end{fact}

\begin{proof}
The case when $L$ is perfect follows by Fact~\ref{fact:kosters}, so we suppose that $L$ is not perfect.
Let $p = \Chara(L)$ and $F = \{ a^p : a \in F\}$.
Note that $|L \setminus F| = |L|$ as $F$ is a proper subfield of $L$.
Suppose that $f$ is not separable.
Then $f(t) = g(t^p)$ for some $g \in L[t]$.
So $f(L) \subseteq F$ hence $|L \setminus f(L)| = |L|$.
Suppose that $f$ is separable.
As $f$ is irreducible there is no $a \in \lins$ such that $f(a) = 0$.
Apply Fact~\ref{fact:finite-image}.
\end{proof}

\section{An extension to diophantine sets}
\label{section:diophantine}
A subset $X$ of $L^m$ is \textbf{diophantine} if there is $f \in L[x_1,\ldots,x_m,y_1,\ldots,y_n]$ such that 
$$ X = \{ (a_1,\ldots,a_m) \in L^m : f(a_1,\ldots,a_m,y_1,\ldots,y_n) = 0 \text{ has a solution in } L \}. $$
The case of Fact~\ref{fact:diophantine} when $m = 1$ is due to Fehm~\cite{Fehm-subfield}, we will see below that the general case follows immediately from this case.\vspace{-.3ex}

\begin{fact}
\label{fact:diophantine}
Suppose that $L$ is perfect and large and $X$ is an infinite diophantine subset of $L^m$.
Then $|X| = |L|$, and if $K$ is a proper subfield of $L$ then $|X \setminus K^m| = |L|$.
\end{fact}\vspace{-.1ex}

Fact~\ref{fact:diophantine} fails if $L$ is not perfect as the image of the Frobenius is a diophantine subfield.
Denef has shown that $\Zz$ is a Diophantine subset of $\Rr(t)$~\cite{denef-diophantine}.
Theorem~\ref{thm:to-come} is proven in \cite{exis-2-paper}.\vspace{-.2ex}

\begin{theorem}
\label{thm:to-come}
Suppose that $L$ is large and perfect and $X \subseteq L^m$ is diophantine.
Then there are closed subvarieties $V_1,\ldots,V_k$ of $\Aa^m_L$ and $X_1,\ldots,X_k$ such that each $X_i$ is an \'etale open subset of $V_i(L)$ and $X = \bigcup_{i = 1}^{k} X_i$.
In particular any diophantine subset of $L$ is a union of an \'etale open subset of $L$ and a finite set.
\end{theorem}\vspace{-.1ex}

Theorem~\ref{thm:to-come} also requires perfection.
Suppose that $L$ is separably closed and not algebraically closed, and let $F$ be the image of the Frobenius $L \to L$.
It is shown in \cite{firstpaper} that the \'etale open topology over $L$ agrees with the Zariski topology, so $F$ is a dense and co-dense subset of $L$.
We now prove Fact~\ref{fact:diophantine}.

\begin{proof}
We only prove the second claim as the proof of the first claim is similar.
Fix a coordinate projection $\pi : L^m \to L$ such that $\pi(X)$ is infinite.
Note that $\pi(X)$ is diophantine and $\pi(X) \setminus L \subseteq \pi(X \setminus L^m)$, so we may suppose that $m = 1$.
By Theorem~\ref{thm:to-come} we have $X = A \cup O$ where $A$ is finite and $O$ is a nonempty \'etale open subset of $L$.
Proposition B yields $|O \setminus K| = |L|$.
\end{proof}

\bibliographystyle{amsalpha}
\bibliography{ref}

\end{document}